\documentclass[final]{siamltex}

\newcommand{\tr}{^{\sf T}}
\newcommand{\m}[1]{{\bf{#1}}}
\newcommand{\g}[1]{\bm #1}
\newcommand{\C}[1]{{\cal {#1}}}

\usepackage{bm}
\usepackage{amsfonts}
\usepackage{graphicx}
\newtheorem{remark}{Remark}

\medskip

\title{{\bf An exact algorithm for graph partitioning}
\thanks{ November 09, 2009.
This material is based upon work supported by the
National Science Foundation under Grant 0620286.
}}
\date{}

\author{
        William W. Hager\thanks{{\tt hager@math.ufl.edu},
        http://www.math.ufl.edu/$\sim$hager,
        PO Box 118105,
        Department of Mathematics,
        University of Florida, Gainesville, FL 32611-8105.
        Phone (352) 392-0281. Fax (352) 392-8357.}
\and
        Dzung T. Phan\thanks{{\tt dphan@math.ufl.edu},
        http://www.math.ufl.edu/$\sim$dphan,
        PO Box 118105,
        Department of Mathematics,
        University of Florida, Gainesville, FL 32611-8105.
        Phone (352) 392-0281. Fax (352) 392-8357.}
\and
        Hongchao Zhang\thanks{{\tt hozhang@math.lsu.edu},
        http://www.math.lsu.edu/$\sim$hozhang,
        Department of Mathematics,
        140 Lockett Hall,
        Center for Computation and Technology,
        Louisiana State University, Baton Rouge, LA 70803-4918.
        Phone (225) 578-1982. Fax (225) 578-4276.}
}

\begin{document}
\maketitle

\begin{abstract}
An exact algorithm is presented for solving edge weighted
graph partitioning problems.
The algorithm is based on a branch and bound
method applied to a continuous quadratic programming formulation of the problem.
Lower bounds are obtained by decomposing the objective function into
convex and concave parts and replacing the concave part by an affine
underestimate.
It is shown that the best affine underestimate
can be expressed in terms of the center and the radius of the
smallest sphere containing the feasible set.
The concave term is obtained either by a constant diagonal shift
associated with the smallest eigenvalue of the objective function Hessian,
or by a diagonal shift obtained by
solving a semidefinite programming problem.
Numerical results show that the proposed algorithm
is competitive with state-of-the-art graph partitioning codes.
\end{abstract}



\pagestyle{myheadings} \thispagestyle{plain}
\markboth{W. W. HAGER, D. T. PHAN, H. ZHANG}
{GRAPH PARTITIONING}

\section{Introduction}
\label{introduction}

Given a graph with edge weights, the
graph partitioning problem is to partition the vertices into
two sets satisfying specified size constraints,
while minimizing the sum of the weights of
the edges that connect the vertices in the two sets.
Graph partitioning problems arise in many areas including
VLSI design, data mining,
parallel computing, and sparse matrix factorizations
\cite{HagerKrylyuk99, Johnson93, Lengauer, Teng}.
The graph partitioning problem is NP-hard \cite{Garey76}.

There are two general classes of methods for the graph partitioning
problem, exact methods which compute the optimal partition,
and heuristic methods which try
to quickly compute an approximate solution.
Heuristic methods include spectral methods
\cite{HendricksonLeland95}, geometric methods
\cite{GilbertMillerTeng98}, multilevel schemes \cite{Hendrickson},
optimization-based methods \cite{FalknerRendlWolkowicz94}, and
methods that employ randomization techniques such as genetic
algorithms \cite{SoperWalshawCross04}.
Software which implements heuristic methods includes
Metis (\cite{KarypisKumar98e,KarypisKumar99b,KarypisKumar00}),
Chaco \cite{hendrickson94chaco}, Party \cite{Preis96theparty},
PaToH \cite{Catalyurek-hyper},
SCOTCH \cite{PellegriniRomanAmestoy00},
Jostle \cite{WalshawCrossEverett97},
Zoltan \cite{Zoltan06ipdps}, and
HUND \cite{GrigoriBomanDonfackDavis08}.

This paper develops an exact algorithm for the graph partitioning problem.
In earlier work,
Brunetta, Conforti, and Rinaldi
\cite{bcr97} propose a branch-and-cut scheme based on a linear
programming relaxation and subsequent cuts based on separation techniques.
A column generation approach is developed by Johnson, Mehrotra, and
Nemhauser \cite{Johnson93}, while Mitchell \cite{Mitchell01}
develops a polyhedral approach.
Karisch, Rendl, and Clausen \cite{Karisch00} develop
a branch-and-bound method utilizing a semidefinite programming
relaxation to obtain a lower bound. Sensen \cite{Sensen01} develops
a branch-and-bound method based on a
lower bound obtained by solving a multicommodity flow problem.

In this paper, we develop a branch-and-bound algorithm based on
a quadratic programming (QP) formulation of the graph partitioning problem.
The objective function of the QP is expressed as the sum of
a convex and a concave function.
We consider two different techniques for making this decomposition,
one based on eigenvalues and the other based on semidefinite programming.
In each case, we give an affine underestimate for the concave function,
which leads to a tractable lower bound in the branch and bound algorithm.

The paper is organized as follows.
In Section \ref{continuousQP} we review the continuous
quadratic programming formulation of the
graph partitioning problem developed in \cite{HagerKrylyuk99}
and we explain how to associate a solution of the continuous problem
with the solution to the discrete problem.
In Section \ref{LowerBound} we discuss approaches for decomposing
the objective function for the QP into the sum of convex and a concave
functions, and in each case, we show how to generate an affine lower bound
for the concave part.
Section \ref{BB} gives the branch-and-bound algorithm,
while Section \ref{NS} provides necessary and sufficient conditions
for a local minimizer.
Section \ref{numerics} compares the performance of the new
branch-and-bound algorithm to earlier results given in
\cite{Karisch00} and \cite{Sensen01}.

{\bf Notation.} Throughout the paper, $\| \cdot \|$ denotes
the Euclidian norm. $\m{1}$ is the vector whose entries are all 1.
The dimension will be clear from context.
If $\m{A} \in \mathbb{R}^{n\times n}$, $\m{A} \succeq \m{0}$ means
that $\m{A}$ is positive semidefinite.
We let $\m{e}_i$ denote the $i$-th column of the identity matrix;
again, the dimension will be clear from context.
If $\C{S}$ is a set, then $|\C{S}|$ is the number of elements in $\C{S}$.
The gradient $\nabla f (\m{x})$ is a row vector.
\section{Continuous quadratic programming formulation}
\label{continuousQP}
Let $G$ be a graph with $n$ vertices
\[
\C{V} = \{ 1, 2, \cdots , n \},
\]
and let $a_{ij}$ be a weight associated with the edge $(i,j)$.
When there is no edge between $i$ and $j$, we set $a_{ij} = 0$.
For each $i$ and $j$, we assume that $a_{ii} = 0$ and $a_{ij} = a_{ji}$;
in other words, we consider an undirected graph without self loops
(a simple, undirected graph).
The sign of the weights is not restricted, and in fact,
$a_{ij}$ could be negative, as it would be in the max-cut problem.
Given integers $l$ and $u$ such that $0 \le l \le u \le n$,
we wish to partition the vertices into two disjoint sets,
with between $l$ and $u$ vertices in one set,
while minimizing the sum of the weights associated with edges
connecting vertices in different sets.
The edges connecting the two sets in the partition are referred to
as the cut edges, and the optimal partition minimizes the sum of the
weights of the cut edges.
Hence, the graph partitioning problem is also called the min-cut problem.

In \cite{HagerKrylyuk99} we show that for a suitable choice of
the diagonal matrix $\m{D}$,
the graph partitioning problem is equivalent to the following
continuous quadratic programming problem:
\begin{equation}\label{Q}
\begin{array}{c}
\mbox{minimize } \; \; f(\m{x}) := (\m{1} - \m{x} )\tr (\m{A}+\m{D}) \m{x} \\
\rule{0in}{.2in} \mbox{ subject to } \; \m{0} \le \m{x} \le \m{1} ,
\;\; l \le \m{1}\tr \m{x} \le u,
\end{array}
\end{equation}
where $\m{A}$ is the matrix with elements $a_{ij}$.
Suppose $\m{x}$ is binary and let us define the sets
\begin{equation}\label{part}
\C{V}_0 = \{ i : x_i = 0 \} \quad \mbox{and} \quad \C{V}_1 = \{ i:
x_i = 1\} .
\end{equation}
It can be checked that $f(\m{x})$ is the sum of the weights of
the cut edges associated with the partition (\ref{part}).
Hence, if we add the restriction that $\m{x}$ is binary, then
(\ref{Q}) is exactly equivalent to finding the partition which
minimizes the weight of the cut edges.
Note, though, that there are no binary constraints in (\ref{Q}).
The equivalence between (\ref{Q}) and the graph partitioning problem
is as follows (see \cite[Thm. 2.1]{HagerKrylyuk99}):
\bigskip

\begin{theorem}
\label{Q=GP}
If the diagonal matrix $\m{D}$ is chosen so that
\begin{equation}\label{d-condition}
d_{ii} + d_{jj} \ge 2a_{ij} \quad \mbox{and} \quad d_{ii} \ge 0
\end{equation}
for each $i$ and $j$, then $(\ref{Q})$ has a binary solution
$\m{x}$ and the partition given by $(\ref{part})$ is a min-cut.
\end{theorem}
\bigskip

The generalization of this result to multiset partitioning is given
in \cite{HagerKrylyuk02}.
The condition (\ref{d-condition}) is satisfied, for example,
by the choice
\[
d_{jj} = \max \;\; \{ 0, a_{1j}, a_{2j}, \ldots , a_{nj} \}
\]
for each $j$.
The proof of Theorem \ref{Q=GP} was based on showing that any solution
to (\ref{Q}) could be transformed to a binary solution without changing
the objective function value.
With a modification of this idea, any feasible point can
be transformed to a binary feasible point without increasing the
objective function value.
We now give a constructive proof of this result,
which is used when we solve (\ref{Q}).

\begin{corollary}
\label{move_to_binary}
If $\m{x}$ is feasible in $(\ref{Q})$ and the diagonal matrix $\m{D}$
satisfies $(\ref{d-condition})$,
then there exists a binary $\m{y}$ with $f(\m{y}) \le f(\m{x})$ and
$y_i = x_i$ whenever $x_i$ is binary.
\end{corollary}

\begin{proof}
We first show how to find $\m{z}$ with the property that
$\m{z}$ is feasible in (\ref{Q}),
$f(\m{z}) \le f (\m{x})$, $\m{1}\tr\m{z}$ is integer,
and the only components of $\m{z}$ and $\m{x}$
which differ are the fractional components of $\m{x}$.
If $\m{1}\tr\m{x} = u$ or $\m{1}\tr\m{x} = l$, then we are done since
$l$ and $u$ are integers;
hence, we assume that $l < \m{1}\tr \m{x} < u$.
If all components of $\m{x}$ are binary, then we are done,
so suppose that there exists a nonbinary component $x_i$.
Since $a_{ii} = 0$, a Taylor expansion of $f$ gives
\[
f(\m{x} + \alpha \m{e}_i) = f (\m{x}) + \alpha \nabla f(\m{x})_i
- \alpha^2 d_{ii} ,
\]
where $\m{e}_i$ is the $i$-th column of the identity matrix.
The quadratic term in the expansion is nonpositive since $d_{ii} \ge 0$.
If the first derivative term is negative, then increase $\alpha$ above 0
until either $x_i + \alpha$ becomes 1 or
$\m{1}\tr\m{x} + \alpha$ is an integer.
Since the first derivative term is negative and $\alpha > 0$,
$f(\m{x} + \alpha \m{e}_i) < f (\m{x})$.
If $\m{1}\tr\m{x} + \alpha$ becomes an integer, then we are done.
If $x_i + \alpha$ becomes 1, then we reach a point $\m{x}_1$
with one more binary component and with an objective function value
no larger than $f(\m{x})$.
If the first derivative term is nonnegative, then decrease $\alpha$ below 0
until either $x_i + \alpha$ becomes 0 or
$\m{1}\tr\m{x} + \alpha$ is an integer.
Since the first derivative term is nonnegative and $\alpha < 0$,
$f(\m{x} + \alpha \m{e}_i) \le f (\m{x})$.
If $\m{1}\tr\m{x} + \alpha$ becomes an integer, then we are done.
If $x_i + \alpha$ becomes 0, then we reach a point $\m{x}_1$
with one more binary component and with a smaller value for the cost function.
In this latter case, we choose another nonbinary component of $\m{x}_1$ and
repeat the process.
Hence, there is no loss of generality in assuming that $\m{1}\tr\m{x}$ is
an integer.

Suppose that $\m{x}$ is not binary.
Since $\m{1}\tr\m{x}$ is an integer,
$\m{x}$ must have at least two nonbinary components,
say $x_i$ and $x_j$.
Again, expanding $f$ is a Taylor series gives
\[
f(\m{x} + \alpha (\m{e}_i-\m{e}_j)) = f (\m{x})
+ \alpha (\nabla f(\m{x})_i - \nabla f(\m{x})_j) +
\alpha^2 (2a_{ij}- d_{ii} -d_{jj}) .
\]
By (\ref{d-condition}), the quadratic term is
nonpositive for any choice of $\alpha$.
If the first derivative term is negative, then we
increase $\alpha$ above 0 until either $x_i + \alpha$
reaches 1 or $x_j - \alpha$ reach 0.
Since the first derivative term is negative and $\alpha > 0$,
we have $f(\m{x} + \alpha (\m{e}_i-\m{e}_j)) < f (\m{x})$.
If the first derivative term is nonnegative, then we
decrease $\alpha$ below 0 until either $x_i + \alpha$
reaches 0 or $x_j - \alpha$ reach 1.
Since the first derivative term is nonnegative and $\alpha < 0$,
it follows that $f(\m{x} + \alpha (\m{e}_i-\m{e}_j)) \le f (\m{x})$.
In either case, the value of the cost function does not increase,
and we reach a feasible point $\m{x}_1$ with $\m{1}\tr\m{x}_1$
integer and with at least one more binary component.
If $\m{x}_1$ is not binary, then $\m{x}_1$ must have at
least two nonbinary components; hence, the adjustment
process can be continued until all the components of $\m{x}$ are binary.
These adjustments to $\m{x}$ do not increase the value of the cost
function and we only alter the fractional components of $\m{x}$.
This completes the proof.
\end{proof}
\smallskip

\section{Convex lower bounds for the objective function}
\label{LowerBound}
We compute an exact solution to the continuous formulation (\ref{Q})
of graph partitioning problem using a branch and bound algorithm.
The bounding process requires a lower bound for the objective
function when restricted to the intersection of a box and two half spaces.
This lower bound is obtained by
writing the objective function as the sum of a convex and a concave function
and by replacing the concave part by the best affine underestimate.
Two different strategies are given for decomposing the objective function.


\subsection{Lower bound based on minimum eigenvalue}
\label{mineig}
Let us decompose the objective function
$f(\m{x}) = (\m{1} - \m{x} )\tr (\m{A}+\m{D}) \m{x}$ in the
following way:
\[
f(\m{x}) = (f(\m{x}) + \sigma \|\m{x}\|^2) - \sigma \|\m{x}\|^2,
\]
where $\sigma$ is the maximum of 0 and the largest eigenvalue
of $\m{A} + \m{D}$.
This represents a DC (difference convex) decomposition (see \cite{HPT95})
since $f(\m{x}) + \sigma \|\m{x}\|^2$ and $\sigma \|\m{x}\|^2$ are both convex.
The concave term $- \|\m{x}\|^2$
is underestimated by an affine function $\ell$ to obtain
a convex underestimate $f_{L}$ of $f$ given by
\begin{equation}\label{fL1}
f_{L}(\m{x}) = \left( f (\m{x}) + \sigma \|\m{x}\|^2 \right) + \sigma \ell
(\m{x}). \nonumber
\end{equation}
We now consider the problem of finding the best
affine underestimate $\ell$ for the concave function
$-\|\m{x}\|^2$ over a given compact, convex set denoted $\C{C}$.
The set of affine underestimators for $-\|\m{x}\|^2$ is given by
\[
\C{S}_1 = \{ \ell: \mathbb{R}^n \rightarrow \mathbb{R}
\mbox{ such that } \ell
\mbox{ is affine and } - \|\m{x}\|^2 \ge \ell(\m{x})
\mbox{ for all } \m{x} \in \C{C} \} .
\]
The best affine underestimate is a solution of the problem
\begin{equation}\label{linearest}
\min_{\ell \in \C{S}_1} \;\; \max_{\m{x}\in \C{C}} \;\; - \left(
\|\m{x}\|^2 + \ell (\m{x}) \right) .
\end{equation}
The following result generalizes Theorem 3.1 in \cite{HagerPhan09}
where we determine the best affine underestimate
for $-\|\m{x}\|^2$ over an ellipsoid.
\smallskip

\begin{theorem}
\label{UnderTheorem1}
Let $\C{C} \subset \mathbb{R}^n$ be a compact, convex set and
let $\m{c}$ be the center and $r$ be the radius of the smallest
sphere containing $\C{C}$.
This smallest sphere is unique and a solution of $(\ref{linearest})$ is
\[
\ell^* (\m{x}) = -2\m{c}\tr\m{x} + \|\m{c}\|^2 - r^2 .
\]
Furthermore,
\[
\min_{\ell \in \C{S}_1} \;\; \max_{\m{x}\in \C{C}} \;\; - \left(
\|\m{x}\|^2 + \ell^* (\m{x}) \right) = r^2.
\]
\end{theorem}
\smallskip

\begin{proof}
To begin, we will show that the minimization in (\ref{linearest})
can be restricted to a compact set. Clearly, when carrying out the
minimization in (\ref{linearest}), we should restrict our attention
to those $\ell$ which touch the function $h(\m{x}) := -\|\m{x}\|^2$
at some point in $\C{C}$. Let $\m{y} \in \C{C}$ denote the point of
contact. Since $h(\m{x}) \ge \ell (\m{x})$ and $h (\m{y}) = \ell
(\m{y})$, a lower bound for the error $h (\m{x}) - \ell (\m{x})$
over $\m{x} \in \C{C}$ is
\[
h (\m{x}) - \ell (\m{x}) \ge |\ell (\m{x}) - \ell (\m{y})| -
|h(\m{x}) - h(\m{y})| .
\]
If $M$ is the difference between the maximum and minimum value of
$h$ over $\C{C}$, then we have
\begin{equation}\label{change}
h (\m{x}) - \ell (\m{x}) \ge |\ell (\m{x}) - \ell (\m{y})| - M.
\end{equation}

An upper bound for the minimum in (\ref{linearest}) is obtained by
the linear function $\ell_0$ which is constant on $\C{C}$, with
value equal to the minimum of $h(\m{x})$ over $\m{x} \in \C{C}$. If
$\m{w}$ is a point where $h$ attains its minimum over $\C{C}$, then
we have
\[
\max_{\m{x} \in \C{C}} \;\; h (\m{x}) - \ell_0 (\m{x}) = \max_{\m{x} \in
\C{C}} \;\; h (\m{x}) - h (\m{w}) = M.
\]
Let us restrict our attention to the linear functions $\ell$
which achieve an
objective function value in (\ref{linearest}) which is at least as
small as that of $\ell_0$.
For these $\ell$ and for $\m{x} \in \C{C}$, we have
\begin{equation}\label{ell0}
h (\m{x}) - \ell (\m{x}) \le \max_{\m{x}\in \C{C}} \;\; h (\m{x}) - \ell
(\m{x}) \le \max_{\m{x}\in \C{C}} \;\; h (\m{x}) - \ell_0 (\m{x}) = M .
\end{equation}
Combining (\ref{change}) and (\ref{ell0}) gives
\begin{equation}\label{2M}
|\ell (\m{x}) - \ell (\m{y})| \le 2M .
\end{equation}
Thus, when we carry out the minimization in (\ref{linearest}), we
should restrict our attention to linear functions which touch $h$ at some point
$\m{y} \in \C{C}$ and with the change in $\ell$ across $\C{C}$
satisfying the bound (\ref{2M}) for all $\m{x}\in \C{C}$. This tells
us that the minimization in (\ref{linearest}) can be restricted to a
compact set, and that a minimizer must exist.

Suppose that $\ell$ attains the minimum in (\ref{linearest}). Let
$\m{z}$ be a point in $\C{C}$ where $h (\m{x}) - \ell (\m{x})$
achieves its maximum. A Taylor expansion around $\m{x} = \m{z}$
gives
\[
h (\m{x}) - \ell (\m{x}) = h (\m{z}) - \ell (\m{z})  + (\nabla
h(\m{z}) - \nabla \ell)(\m{x}-\m{z}) - \|\m{x} - \m{z}\|^2 .
\]
Since $\ell \in \C{S}_1$,
$h (\m{x}) - \ell (\m{x}) \ge 0$ for all $\m{x} \in \C{C}$.
It follows that
\begin{equation}\label{flx}
 h (\m{z}) - \ell (\m{z})  \ge -(\nabla h(\m{z}) - \nabla \ell)(\m{x}-\m{z})
+ \|\m{x} - \m{z}\|^2 .
\end{equation}
Since $\C{C}$ is convex,
the first-order optimality conditions for $\m{z}$ give
\[
(\nabla h(\m{z}) - \nabla \ell)(\m{x}-\m{z}) \le 0
\]
for all $\m{x} \in \C{C}$.
It follows from (\ref{flx}) that
\begin{equation}\label{diameter}
h (\m{z}) - \ell (\m{z})  \ge \|\m{x} - \m{z}\|^2
\end{equation}
for all $\m{x} \in \C{C}$.
There exists $\m{x} \in \C{C}$ such that
$\|\m{x} - \m{z}\| \ge r$
or else $\m{z}$ would be the center of a smaller sphere
containing $\C{C}$.
Hence, (\ref{diameter}) implies that
\[
h (\m{z}) - \ell (\m{z})  \ge r^2.
\]
It follows that
\begin{equation}\label{l_lower}
\max_{\m{x} \in \C{C}} \;\; h(\m{x}) - \ell (\m{x}) \ge
h(\m{z}) - \ell (\m{z}) \ge r^2.
\end{equation}

We now observe that for the specific linear function $\ell^*$ given
in the statement of the theorem, (\ref{l_lower}) becomes an
equality, which implies the optimality of $\ell^*$ in
(\ref{linearest}).
Expand $h$ in a Taylor series around $\m{x} = \m{c}$ to obtain
\begin{eqnarray}
h(\m{x}) &=& -\|\m{c}\|^2 - 2\m{c}\tr(\m{x}-\m{c}) -\|\m{x}-\m{c}\|^2
\nonumber \\
&=& -2\m{c}\tr\m{x} + \|\m{c}\|^2 - \|\m{x}-\m{c}\|^2.  \nonumber
\end{eqnarray}
Subtract $\ell^* (\m{x}) = -2\m{c}\tr\m{x} + \|\m{c}\|^2 - r^2$ from
both sides to obtain
\begin{equation}\label{h-l}
h(\m{x}) - \ell^* (\m{x}) = r^2 - \|\m{x}-\m{c}\|^2 .
\end{equation}
If $\m{c} \in \C{C}$, then the maximum in (\ref{h-l})
over $\m{x} \in \C{C}$ is attained by $\m{x} = \m{c}$ for which
\[
h(\m{c}) - \ell^* (\m{c}) = r^2.
\]
Consequently, (\ref{l_lower}) becomes an equality for $\ell =
\ell^*$, which implies the optimality of $\ell^*$ in
(\ref{linearest}).

We can show that $\m{c} \in \C{C}$ as follows:
Suppose $\m{c} \not\in \C{C}$.
Since $\C{C}$ is compact and convex, there exists a hyperplane $\C{H}$
strictly separating $\m{c}$ and $\C{C}$ -- see Figure \ref{c_in_C}
\begin{figure}
\includegraphics[scale=.4]{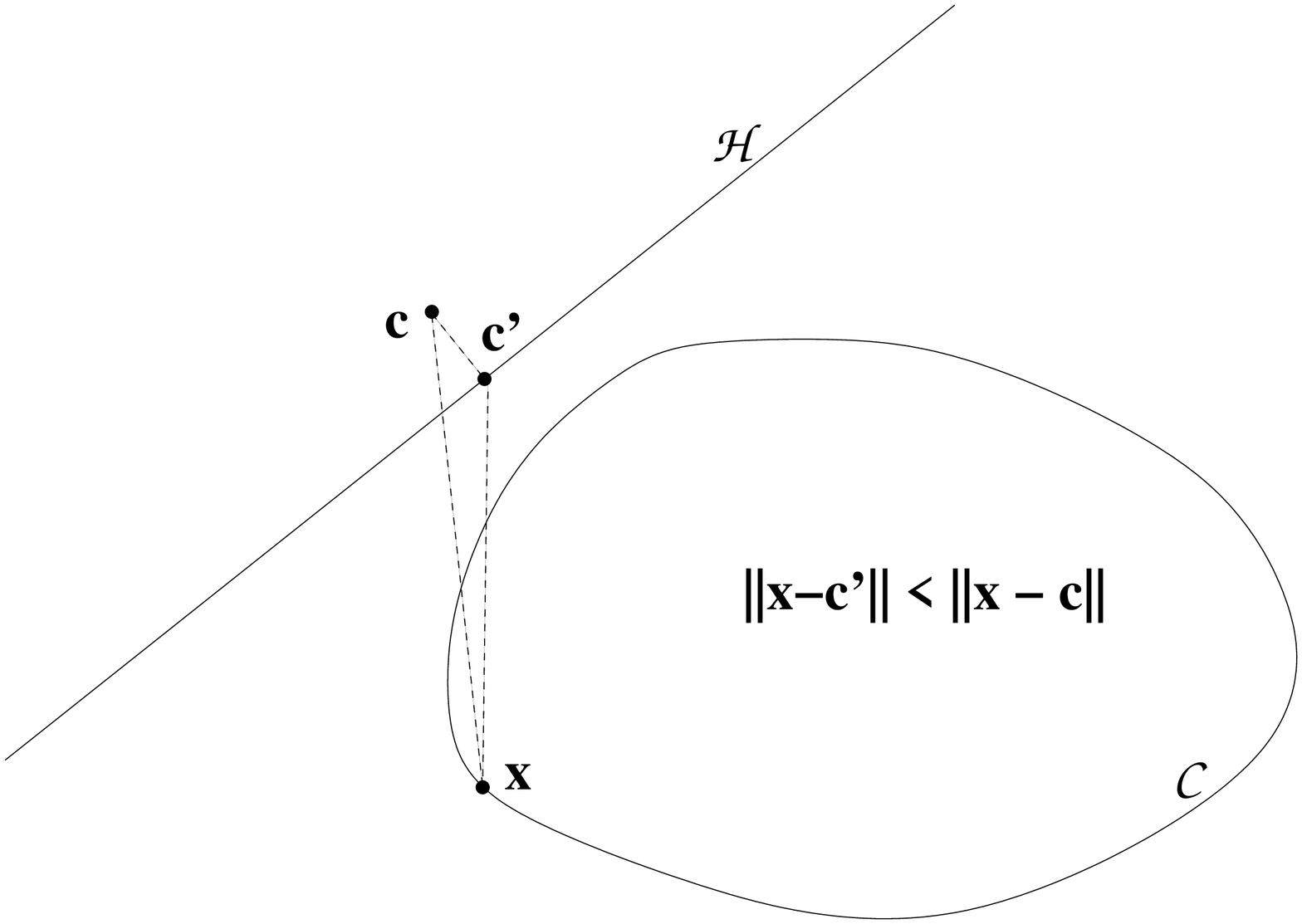}
\caption{Suppose $\m{c} \not\in \C{C}$}
\label{c_in_C}
\end{figure}
If $\m{c}'$ is the projection of $\m{c}$ onto $\C{H}$,
then
\begin{equation}\label{circle}
\|\m{x} - \m{c}'\| < \|\m{x} - \m{c}\| \quad \mbox{for all }\m{x} \in \C{C} .
\end{equation}
Let $\m{x}' \in \C{C}$ be the point
which is farthest from $\m{c}'$ and
let $\m{x} \in \C{C}$ be the point farthest from $\m{c}$.
Hence, $\|\m{x} - \m{c}\| = r$.
By (\ref{circle}), we have
$\|\m{x}' - \m{c}'\| < \|\m{x} - \m{c}\| = r$;
it follows that the
sphere with center $\m{c}'$ and radius $\|\m{x}' - \m{c}'\|$ contains
$\C{C}$ and has radius smaller than $r$.
This contradicts the assumption that $r$ was the sphere of smallest
radius containing $\C{C}$.

The uniqueness of the smallest sphere containing $\C{C}$ is as follows:
Suppose that there exist two different smallest spheres
$\C{S}_1$ and $\C{S}_2$ containing $\C{C}$.
Let $\C{S}_3$ be the smallest sphere containing
$\C{S}_1 \cap \C{S}_2$.
Since the diameter of the intersection is strictly less than the
diameter of $\C{S}_1$ or $\C{S}_2$, we contradict the assumption that
$\C{S}_1$ and $\C{S}_2$ were spheres of smallest radius containing $\C{C}$.
\end{proof}

\begin{remark}
\label{rem0}
Although the smallest sphere containing $\C{C}$ in Theorem
\ref{UnderTheorem1} is unique, the best linear underestimator of
$h(\m{x}) = -\|\m{x}\|^2$ is not unique.
For example, suppose $\m{a}$ and $\m{b} \in {\mathbb{R}}^n$ and
$\C{C}$ is the line segment
\[
\C{C} = \{ \m{x} \in \mathbb{R}^n : \m{x} = \alpha \m{a} + (1-\alpha)\m{b},
\quad \alpha \in [0, 1] \} .
\]
Along this line segment, $h$ is a concave quadratic in one variable.
The best affine underestimate along the line segment corresponds to
the line connecting the ends of the quadratic restricted to the line segment.
Hence, in $\mathbb{R}^{n+1}$,
any hyperplane which contains the points $(h(\m{a}), \m{a})$ and
$(h(\m{b}), \m{b})$ leads to a best affine underestimate.
\end{remark}

\begin{remark}
\label{rem1}
Let $\C{C}$ be the box
\[
\C{B} = \{ \m{x} \in \mathbb{R}^n : \m{p} \le \m{x} \le \m{q} \}.
\]
The diameter of $\C{B}$, the distance between the points in $\C{B}$
with greatest separation, is $\|\m{p} - \m{q}\|$.
Hence, the smallest sphere containing $\C{B}$ has radius at least
$\|\m{p} - \m{q}\|/2$.
If $\m{x} \in \C{B}$, then
\[
|x_i - (p_i + q_i)/2| \le (q_i - p_i)/2
\]
for every $i$.
Consequently, $\|\m{x} - (\m{p} + \m{q})/2\| \le \|\m{p} - \m{q}\|/2$ and
the sphere with center
$\m{c} = (\m{p}+\m{q})/2$ and radius $r = \|\m{p} - \m{q}\|/2$
contains $\C{B}$.
It follows that this is the smallest sphere containing $\C{B}$ since
any other sphere must have radius at least $\|\m{p} - \m{q}\|/2$.
\end{remark}

\begin{remark}
\label{rem2}
Finding the smallest sphere containing $\C{C}$ may not be easy.
However, the center and radius of any sphere containing $\C{C}$
yields an affine underestimate for $\|\m{x}\|^2$ over $\C{C}$.
That is, if $\C{S}$ is a sphere with $\C{C} \subset \C{S}$,
then the best affine underestimate for $-\|\m{x}\|^2$ over
$\C{S}$ is also an affine underestimate for $-\|\m{x}\|^2$ over $\C{C}$.
\end{remark}

%

\subsection{Lower bound based on semidefinite programming}
\label{SDP}

A different DC decomposition of
$f(\m{x}) = (\m{1} - \m{x} )\tr (\m{A}+\m{D}) \m{x}$ is the following:
\[
f(\m{x}) = (f(\m{x}) + \m{x}\tr\g{\Lambda}\m{x}) - \m{x}\tr\g{\Lambda}\m{x} ,
\]
where $\g{\Lambda}$ is a diagonal matrix with $i$-th
diagonal element $\lambda_i \ge 0$.
We would like to make the second term
$\m{x}\tr\g{\Lambda}\m{x}$ as small as possible while keeping
the first term $f(\m{x}) + \m{x}\tr\g{\Lambda}\m{x}$ convex.
This suggests the following semidefinite programming problem
\begin{equation}\label{sdp}
\begin{array}{c}
\mbox{minimize } \; \; \sum_{i=1}^n \lambda_i \\
\rule{0in}{.2in} \mbox{ subject to } \; \g{\Lambda} -(\m{A} + \m{D})
\succeq \m{0}, \quad \g{\Lambda} \succeq \m{0},
\end{array}
\end{equation}
where $\g{\lambda}$ is the diagonal of $\g{\Lambda}$.
If the diagonal of $\m{A} + \m{D}$ is nonnegative,
then the inequality $\g{\Lambda} \succeq \m{0}$ can be dropped
since it is implied by the inequality
$\g{\Lambda} -(\m{A} + \m{D}) \succeq \m{0}$.

As before, we seek the best linear underestimate of the concave
function $- \m{x}\tr\g{\Lambda}\m{x}$ over a compact, convex set $\C{C}$.
If any of the $\lambda_i$ vanish,
then reorder the components of $\m{x}$ so that
$\m{x} = (\m{y}, \m{z})$ where $\m{z}$ corresponds
to the components of $\lambda_i$ that vanish.
Let $\g{\Lambda}_+$ be the principal submatrix of $\g{\Lambda}$ corresponding
to the positive diagonal elements, and define the set
\[
\C{C}_+ = \{ \m{y} : (\m{y},\m{z}) \in \C{C} \mbox{ for some } \m{z} \} .
\]
The problem of finding the best linear underestimate for
$- \m{x}\tr\g{\Lambda}\m{x}$ over $\C{C}$ is essentially equivalent to
finding the best linear underestimate for
$-\m{y}\tr\g{\Lambda}_+\m{y}$ over the $\C{C}_+$.
Hence, there is no loss of generality in assuming that the diagonal
of $\g{\Lambda}$ is strictly positive.
As a consequence of Theorem \ref{UnderTheorem1}, we have
\smallskip

\begin{corollary}
\label{lambda_under}
Suppose the diagonal of $\g{\Lambda}$ is strictly positive and
let $\m{c}$ be the center and $r$ the radius of the unique smallest
sphere containing the set
\[
\g{\Lambda}^{1/2}\C{C} := \{ \g{\Lambda}^{1/2}\m{x}: \m{x} \in \C{C} \}.
\]
The best linear underestimate of
$-\m{x}\tr\g{\Lambda}\m{x}$ over the compact, convex set $\C{C}$ is
\[
\ell^* (\m{x}) =
-2\m{c}\tr\g{\Lambda}^{1/2}\m{x} + \|\m{c}\|^2 - r^2.
\]
Furthermore,
\[
\min_{\ell \in \C{S}_2} \;\; \max_{\m{x}\in \C{C}} \;\; - \left(
\m{x}\tr\g{\Lambda}\m{x} + \ell^* (\m{x}) \right) = r^2,
\]
where
\[
\C{S}_2 = \{ \ell: \mathbb{R}^n \rightarrow \mathbb{R} \mbox{ such
that } \ell \mbox{ is affine and } -\m{x}\tr\g{\Lambda}\m{x} \ge
\ell(\m{x}) \mbox{ for all } \m{x} \in \C{C} \} .
\]
\end{corollary}
\smallskip

\begin{proof}
With the change of variables $\m{y} = \g{\Lambda}^{1/2}\m{x}$, an
affine function in $\m{x}$ is transformed to an affine function in
$\m{y}$ and conversely, an affine function in $\m{y}$ is transformed
to an affine function in $\m{x}$. Hence, the problem of finding the
best affine underestimate for $- \m{x}\tr\g{\Lambda}\m{x}$ over
$\C{C}$ is equivalent to the problem of finding the best affine
underestimate for $-\|\m{y}\|^2$ over $\g{\Lambda}^{1/2}\C{C}$.
Apply Theorem \ref{UnderTheorem1} to the transformed problem in
$\m{y}$, and then transform back to $\m{x}$.
\end{proof}

\begin{remark}
\label{rem3}
If $\C{C}$ is the box
$\{ \m{x} \in \mathbb{R}^n : \m{0} \le \m{x} \le \m{1} \}$,
then $\g{\Lambda}^{1/2}\C{C}$ is also a box to which we can apply
the observation in Remark \ref{rem1}.
In particular, we have
\begin{equation}\label{box_center}
\m{c} = \frac{1}{2} \g{\Lambda}^{1/2}\m{1} =
\frac{1}{2} \g{\lambda}^{1/2}
\quad \mbox{and} \quad
r = \|\g{\Lambda}^{1/2}\m{1}\|/2 = \|\g{\lambda}^{1/2}\|/2 .
\end{equation}
Hence, $\|\m{c}\|^2 - r^2 = 0$ and we have
$\ell^* (\m{x}) = - \g{\lambda}\tr \m{x}$.
\end{remark}

\begin{remark}
\label{rem4}
Let us consider the set
\[
\C{C} = \{ \m{x} \in \mathbb{R}^n : \m{0} \le \m{x} \le \m{1},
\quad \m{1}\tr \m{x} = b \},
\]
where $0 < b < n$.
Determining the smallest sphere containing
$\g{\Lambda}^{1/2}\C{C}$ may not be easy.
However, as indicated in Remark \ref{rem2}, any sphere containing
$\g{\Lambda}^{1/2} \C{C}$ yields an underestimate for
$\m{x}\tr\g{\Lambda}\m{x}$.
Observe that
\[
\g{\Lambda}^{1/2} \C{C} =
\{ \m{y} \in \mathbb{R}^n : \m{0} \le \m{y} \le \g{\lambda}^{1/2},
\quad \m{y}\tr {\g{\lambda}^{-1/2}} = b \} .
\]
As observed in Remark \ref{rem3}, the center $\m{c}$ and radius $r$ of
the smallest sphere $\C{S}$ containing the set
\[
\{ \m{y} \in \mathbb{R}^n : \m{0} \le \m{y} \le \g{\lambda}^{1/2} \}
\]
are given in (\ref{box_center}).
The intersection of this sphere with the hyperplane
$\m{y}\tr {\g{\lambda}^{-1/2}} = b$ is a lower dimensional sphere $\C{S}'$
whose center $\m{c}'$ is the projection of $\m{c}$ onto the hyperplane.
$\C{S}'$ contains $\C{C}$ since $\C{C}$ is contained in both the
original sphere $\C{S}$ and the hyperplane.
With a little algebra, we obtain
\[
\m{c}' = \frac{1}{2} \g{\lambda}^{1/2} +
\left( \frac{b - .5n}{\sum_{i=1}^n \lambda_i^{-1}} \right)
\g{\lambda}^{-1/2}.
\]
By the Pythagorean Theorem,
the radius $r'$ of the lower dimensional sphere $\C{S}'$ is
\[
r' = \sqrt{.25 \left( \sum_{i=1}^n \lambda_i \right)
- \frac{(b-.5n)^2}{\sum_{i=1}^n \lambda_i^{-1}}} .
\]
Hence, by Corollary \ref{lambda_under}, an underestimate of
$-\m{x}\tr\g{\Lambda}\m{x}$ is given by
\[
\ell(\m{x}) = -\g{\lambda}\tr\m{x} +
\left( \frac{n - 2b}{\sum_{i=1}^n \lambda_i^{-1}} \right) \m{1}\tr\m{x}
+ \|\m{c}'\|^2 - (r')^2 .
\]
Since $\m{1}\tr\m{x} = b$ when $\m{x} \in \C{C}$, it can be shown,
after some algebra, that $\ell (\m{x}) = - \g{\lambda}\tr\m{x}$
(all the constants in the affine function cancel).
Hence, the affine underestimate $\ell^*$ computed in Remark \ref{rem3} for
the unit box and the affine underestimate $\ell$ computed in this remark
for the unit box intersect the hyperplane $\m{1}\tr\m{x} = b$
are the same.
\end{remark}
\section{Branch and bound algorithm}
\label{BB}
Since the continuous quadratic program (\ref{Q}) has a binary solution,
the branching process in the branch and bound algorithm is
based on setting variables to 0 or 1 and reducing the problem
dimension (we do not employ bisections of the feasible region
as in \cite{HagerPhan09}).
We begin by constructing a linear ordering of the vertices of
the graph according to an estimate for the difficulty in
placing the vertex in the partition.
For the numerical experiments, the order was based on the
total weight of the edges connecting a vertex to the adjacent vertices.
If two vertices $v_1$ and $v_2$ have weights $w_1$ and $w_2$ respectively,
then $v_1$ precedes $v_2$ if $w_1 > w_2$.

Let $v_1$, $v_2$, $\ldots$, $v_n$ denote the ordered vertices.
Level $i$ in the branch and bound tree corresponds to
setting the $v_i$-th component of $\m{x}$ to the values 0 or 1.
Each leaf at level $i$ represents a specific selection of 0 and 1
values for the $v_1$ through $v_i$-th components of $\m{x}$.
Hence, a leaf at level $i$ has a label of the form
\begin{equation}\label{leaf}
\tau = (b_1, b_2, \ldots, b_i), \quad b_j = 0 \mbox{ or } 1
\mbox{ for } 1 \le j \le i.
\end{equation}
Corresponding to this leaf, the value of the $v_j$-th component of
$\m{x}$ is $b_j$ for $1 \le j \le i$.

Let $\C{T}_k$ denote the branch and bound tree at iteration $k$
and let $\C{E}(\C{T}_k)$ denote the leaves in the tree.
Suppose $\tau \in \C{E}(\C{T}_k)$ lies at level $i$ in $\C{T}_k$
as in (\ref{leaf}).
Let $\m{x}_\tau$ denote the vector gotten by removing
components $v_j$, $1 \le j \le i$, from $\m{x}$.
The $v_j$-th component of $\m{x}$
has the pre-assigned binary value $b_j$ for $1 \le j \le i$.
After taking into account these assigned binary values,
the quadratic problem reduces to a lower dimensional problem
in the variable $\m{x}_\tau$ of the form
\[
\begin{array}{c}
\mbox{minimize } \; \; f_\tau (\m{x}_\tau) \\
\rule{0in}{.2in} \mbox{ subject to } \; \m{0} \le \m{x}_\tau \le \m{1} ,
\;\; l_\tau \le \m{1}\tr \m{x}_\tau \le u_\tau,
\end{array}
\]
where
\[
u_{\tau} = u - \sum_{j=1}^i b_j \quad \mbox{and} \quad
l_{\tau} = l - \sum_{j=1}^i b_j .
\]
Using the techniques developed in Section \ref{LowerBound},
we replace $f_\tau$ by a convex lower bound denoted $f_{\tau}^L$ and
we consider the convex problem
\begin{equation}\label{Ltau}
\begin{array}{c}
\mbox{minimize } \; \; f_{\tau}^L (\m{x}_\tau) \\
\rule{0in}{.2in} \mbox{ subject to } \; \m{0} \le \m{x}_\tau \le \m{1} ,
\;\; l_\tau \le \m{1}\tr \m{x}_\tau \le u_\tau.
\end{array}
\end{equation}
Let $M(\tau)$ denote the optimal objective function value for (\ref{Ltau}).
At iteration $k$, the leaf $\tau \in \C{E}(\C{T}_k)$ for which
$M(\tau)$ is smallest is used to branch to the next level.
If $\tau$ has the form (\ref{leaf}), then the branching processes
generates the two new leaves
\begin{equation}\label{bisect}
(b_1, b_2, \ldots, b_i, 0) \quad \mbox{and} \quad
(b_1, b_2, \ldots, b_i, 1) .
\end{equation}
An illustration involving a 3-level branch and bound tree
appears in Figure~\ref{bbtree}.
\begin{figure}
\includegraphics[scale=.4]{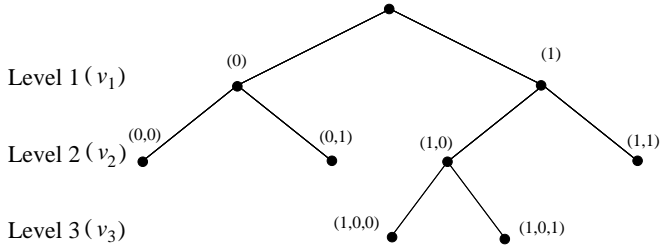}
\caption{Branch and bound tree}
\label{bbtree}
\end{figure}

During the branch and bound process, we must also compute
an upper bound for the minimal objective function value in (\ref{Q}).
This upper bound is obtained using a heuristic technique based
on the gradient projection algorithm and sphere approximations
to the feasible set.
These heuristics for generating an upper bound will be described
in a separate paper.
As pointed out earlier, many heuristic techniques are available
(for example, Metis
(\cite{KarypisKumar98e,KarypisKumar99b,KarypisKumar00}), Chaco
\cite{hendrickson94chaco}, and Party \cite{Preis96theparty}).
An advantage of our quadratic programming based heuristic
is that we start at the solution to the lower bounding problem,
a solution which typically has fractional entries and which is
a feasible starting point for (\ref{Q}).
Consequently, the upper bound is no larger than the objective
function value associated with the optimal point in
the lower-bound problem.
\bigskip

\begin{itemize}
\item[] \hspace{-.2in}\textbf{Convex quadratic branch and bound (CQB)}

\item [1.]
Initialize $\C{T}_0 = \emptyset$ and $k = 0$.
Evaluate both a lower bound for the solution to (\ref{Q}) and
an upper denoted $U_0$.

\item [2.]
Choose $\tau_k \in \C{E}(\C{T}_k)$ such that
$M (\tau_k) =  \min \{ M (\tau) : \tau \in \C{E}(\C{T}_k) \}$.
If $M(\tau_k) = U_k$, then stop, an optimal solution of (\ref{Q})
has been found.

\item[3.]
Assuming that $\tau_k$ has the form (\ref{leaf}), let
$\C{T}_{k+1}$ be the tree obtained by branching at
$\tau_k$ and adding two new leaves as in (\ref{bisect});
also see Figure~\ref{bbtree}.
Evaluate lower bounds for the quadratic programming problems
(\ref{Ltau}) associated with the two new leaves,
and evaluate an improved upper bound, denoted $U_{k+1}$, by using
solutions to the lower bound problems as starting guesses in
a descent method applied to (\ref{Q}).

\item[4.]
Replace $k$ by $k+1$ and return to step 2.
\end{itemize}
\bigskip

Convergence is assured since there are a finite number of
binary values for the components of $\m{x}$.
In the worst case, the branch and bound algorithm will build
all $2^{n+1} - 1$ nodes of the tree.
\smallskip

\section{Necessary and sufficient optimality conditions}
\label{NS}

We use the gradient projection algorithm
to obtain an upper bound for a solution to (\ref{Q}).
Since the gradient projection algorithm can terminate at
a stationary point, we need to be able to distinguish
between a stationary point and a local minimizer, and at a
stationary point which is not a local minimizer,
we need a fast way to compute a descent direction.

We begin by stating the first-order optimality conditions.
Given a scalar $\lambda$, define the vector
\[
\g{\mu}(\m{x},\lambda) = (\m{A}+\m{D})\m{1} - 2(\m{A}+\m{D})\m{x} +
\lambda \m{1},
\]
and the set-valued maps $\C{N} : \mathbb{R} \rightarrow 2^{\mathbb{R}}$
and $\C{M} : \mathbb{R} \rightarrow 2^{\mathbb{R}}$
\[
{\cal N}(\nu) =
\left\{ \begin{array}{cl}
\mathbb{R}  & \mbox{if} \;\; \nu = 0\\
\{ 1 \} & \mbox{if} \;\; \nu < 0\\
\{ 0 \} & \mbox{if} \;\; \nu > 0
\end{array}
\right. ,
\quad
{\cal M}(\nu) =
\left\{ \begin{array}{cl}
\mathbb{R}  & \mbox{if} \;\; \nu = 0\\
\{ u \} & \mbox{if} \;\; \nu > 0\\
\{ l \} & \mbox{if} \;\; \nu < 0
\end{array}
\right.  .
\]
For any vector $\g{\mu}$, $\C{N}(\g{\mu})$ is a vector of sets
whose $i$-component is the set $\C{N}(\mu_i)$.
The first-order optimality (Karush-Kuhn-Tucker) conditions
associated with a local minimizer $\m{x}$ of (\ref{Q}) can be
written in the following way: For some scalar $\lambda$, we have
\begin{equation} \label{KT}
\m{0} \le \m{x} \le \m{1} ,
\quad \m{x} \in {\cal N}(\g{\mu}(\m{x},\lambda)) ,
\quad l \le \m{1}\tr\m{x} \le u , \quad
\mbox{and} \quad \m{1}\tr\m{x} \in \C{M}(\lambda) .
\end{equation}
The first and third conditions in (\ref{KT}) are the constraints in
(\ref{Q}), while the remaining two conditions correspond to
complementary slackness and stationarity of the Lagrangian.

In \cite{HagerKrylyuk99} we give a necessary and sufficient optimality
conditions for (\ref{Q}), which we now review.
Given any $\m{x}$ that is feasible in (\ref{Q}), let us define the sets
\[
{\cal U}(\m{x}) = \{ i : x_i = 1 \}, \quad
{\cal L}(\m{x}) = \{ i : x_i = 0 \}, \quad
\mbox{and}\quad \C{F}(\m{x}) = \{ i: 0 < x_i < 1\} .
\]
We also introduce subsets ${\cal U}_0$ and ${\cal L}_0$ defined by
\[
{\cal U}_0(\m{x},\lambda) =
\{ i \in {\cal U}(\m{x}) : \mu_i(\m{x},\lambda) = 0 \}
\quad \mbox{and} \quad
{\cal L}_0(\m{x},\lambda) =
\{ i \in {\cal L}(\m{x}) : \mu_i(\m{x},\lambda) = 0 \} .
\]
\smallskip

\begin{theorem}
\label{opttheorem}
Suppose that $l = u$ and $\m{D}$ is chosen so that
\begin{equation}\label{a-condition}
d_{ii} + d_{jj} \ge 2a_{ij} .
\end{equation}
for all $i$ and $j$.
A necessary and sufficient condition for $\m{x}$ to be a local minimizer in
$(\ref{Q})$ is that the following all hold:
\begin{itemize}
\item[{\rm (P1)}]
For some $\lambda$,
the first-order conditions $(\ref{KT})$ are satisfied at $\m{x}$.
\item[{\rm (P2)}]
For each $i$ and $j \in {\cal F}(\m{x})$,
we have $d_{ii} + d_{jj} = 2a_{ij}$.
\item[{\rm (P3)}]
Consider the three sets ${\cal U}_0(\m{x},\lambda)$,
${\cal L}_0(\m{x},\lambda)$, and ${\cal F}(\m{x})$. For each $i$ and $j$ in
two different sets, we have $d_{ii} + d_{jj} = 2a_{ij}$.
\end{itemize}
\end{theorem}
\smallskip

In treating the situation $l < u$, an additional condition concerning
the dual multipliers $\lambda$ and $\g{\mu}$ in the first-order
optimality conditions (\ref{KT}) enters into the statement of the result:
\begin{itemize}
\item[{\rm (P4)}]
{\it If $\lambda = \mu_i(\m{x},0) = 0$
for some $i$, then $d_{ii} = 0$ in any of the following three cases:}
\begin{itemize}
\item[{\rm (a)}]
$l < \m{1}\tr \m{x} < u$.
\item[{\rm (b)}]
$x_i > 0$ and $\m{1}\tr \m{x} = u$.
\item[{\rm (c)}]
$x_i < 1$ and $\m{1}\tr \m{x} = l$.
\end{itemize}
\end{itemize}
\bigskip

\begin{corollary}
\label{optcorollary}
Suppose that $l < u$ and $\m{D}$ is chosen so that
\begin{equation}\label{strong-d-condition}
d_{ii} + d_{jj} \ge 2a_{ij} \quad \mbox{and} \quad
d_{ii} \ge 0
\end{equation}
for all $i$ and $j$.
A necessary and sufficient condition for $\m{x}$ to be a local minimizer in
$(\ref{Q})$ is that {\rm (P1)}--{\rm (P4)} all hold.
\end{corollary}
\bigskip

Based on Theorem \ref{opttheorem} and
Corollary \ref{optcorollary}, we can easily check whether
a given stationary point is a local minimizer.
This is in contrast to the general quadratic programming problem
for which deciding whether a given point is a
local minimizer is NP-hard (see \cite{murty1987,pardalos91}).
We now observe that when $\m{x}$ is a stationary point and when any
of the conditions (P2)--(P4) are violated, then a descent direction
is readily available.
\bigskip

\begin{proposition}
\label{descent_direction}
Suppose that $\m{x}$ is a stationary point for $(\ref{Q})$ and
$(\ref{strong-d-condition})$ holds.
If either {\rm (P2)} or {\rm (P3)} is violated, then
$\m{d} = \m{e}_i - \m{e}_j$, with an appropriate choice of
sign, is a descent direction.
If $l < u$, $\lambda  = 0 = \mu_i(\m{x},0)$, and $d_{ii} > 0$,
then $\m{d} = \m{e}_i$, with an appropriate choice of sign,
is a descent direction in any of the cases {\rm (a)--(c)} of {\rm (P4)}.
\end{proposition}

\begin{proof}
The Lagrangian $L$ associated with (\ref{Q}) has the form
\begin{equation}\label{lagrangian}
L(\m{x}) = f(\m{x}) + \lambda (\m{1}\tr\m{x} - b)
- \sum_{i \in \C{L}} \mu_i x_i - \sum_{i \in \C{U}} \mu_i(x_i - 1) ,
\end{equation}
where $b = u$ if $\lambda > 0$, $b = l$ if $\lambda < 0$,
and $\g{\mu}$ stands for $\g{\mu}(\m{x}, \lambda)$.
The sets $\C{L}$ and $\C{U}$ denote $\C{L}(\m{x})$ and
$\C{U}(\m{x})$ respectively.
By the first-order optimality conditions (\ref{KT}), we have
$L(\m{x}) = f(\m{x})$ and $\nabla L(\m{x}) = \m{0}$.
Expanding the Lagrangian around $\m{x}$ gives
\[
L(\m{x}+\m{y}) = L(\m{x}) + \nabla L(\m{x})\m{y} + \frac{1}{2}
\m{y}\tr \nabla^2 L(\m{x})\m{y} = f(\m{x}) - \m{y}\tr (\m{A}+\m{D}) \m{y} .
\]
We substitute for $L$ using (\ref{lagrangian}) to obtain
\begin{eqnarray}
f(\m{x}+\m{y}) &=& L(\m{x}+\m{y})
- \lambda(\m{1}\tr(\m{x}+\m{y}) - b)
+ \sum_{i\in{\cal L}} \mu_i(x_i+y_i)
+ \sum_{i\in{\cal U}} \mu_i(x_i +y_i-1)
\nonumber \\
&=& f(\m{x}) - \lambda \m{1}\tr \m{y}
- \m{y}\tr (\m{A}+\m{D}) \m{y}
+ \sum_{i\in{\cal L}} \mu_i y_i + \sum_{i\in{\cal U}} \mu_i y_i . \label{T}
\end{eqnarray}
If (P2) is violated, then
there are indices $i$ and $j \in \C{F}(\m{x})$ such that
$d_{ii} + d_{jj} > 2a_{ij}$.
We insert $\m{y} = \alpha(\m{e}_i - \m{e}_j)$ in (\ref{T}) to obtain
\begin{equation}\label{h53}
f(\m{x} + \alpha(\m{e}_i - \m{e}_j)) = f(\m{x})
+ \alpha^2(2a_{ij} - d_{ii} - d_{jj}) .
\end{equation}
Since the coefficient of $\alpha^2$ is negative,
$\m{d} = \m{e}_i - \m{e}_j$ is a descent direction for the
objective function.
Since $0 < x_i < 1$ and $0 < x_j < 1$, feasibility is preserved
for $\alpha$ sufficiently small.
In a similar manner, if (P3) is violated by indices $i$ and $j$,
then (\ref{h53}) again holds and
$\m{d} = \pm(\m{e}_i - \m{e}_j)$ is again a descent direction where
the sign is chosen appropriately to preserve feasibility.
For example, if $i \in \C{L}_0(\m{x})$ and $j \in \C{U}_0(\m{x})$,
then $x_i = 0$ and $x_j = 1$.
Consequently, $\m{x} + \alpha (\m{e}_i - \m{e}_j)$ is feasible if $\alpha > 0$
is sufficiently small.

Finally, suppose that
$l < u$, $\lambda  = 0 = \mu_i(\m{x},0)$, and $d_{ii} > 0$.
Substituting $\m{y} = \alpha\m{e}_i$ in (\ref{T}) yields
\[
f(\m{x} + \alpha\m{e}_i) = f(\m{x}) - \alpha^2 d_{ii} .
\]
Since the coefficient $d_{ii}$ of $\alpha^2$ is positive,
$\m{d} = \pm \m{e}_i$ is a descent direction.
Moreover, in any of the cases (a)--(c) of (P4),
$\m{x} + \alpha \m{d}$ is feasible for some $\alpha > 0$ with an
appropriate choice of the sign of $\m{d}$.
\end{proof}
\smallskip

We now give a necessary and sufficient condition for a local
minimizer to be strict.
When a local minimizer is not strict, it may be possible to move
to a neighboring point which has the same objective function value
but which is not a local minimizer.
\smallskip

\begin{corollary}
\label{strict_cor} If $\m{x}$ is a local minimizer for $(\ref{Q})$
and $(\ref{strong-d-condition})$ holds, then $\m{x}$  is a strict
local minimizer if and only if the following conditions hold:
\begin{itemize}
\item[{\rm (C1)}]
${\cal F}(\m{x})$ is empty.
\item[{\rm (C2)}]
$\nabla f (\m{x})_i > \nabla f (\m{x})_j$
for every $i \in \C{L}(\m{x})$ and $j \in \C{U}(\m{x})$.
\item[{\rm (C3)}]
If $l < u$, the first-order optimality conditions $(\ref{KT})$
hold for $\lambda = 0$,
and $\C{Z} := \{ i : \nabla f (\m{x})_i = 0 \} \ne \emptyset$, then either
\begin{itemize}
\item[{\rm (a)}]
$\m{1}\tr\m{x} = u$ and $x_i = 0$ for all $i \in \C{Z}$ or
\item[{\rm (b)}]
$\m{1}\tr\m{x} = l$ and $x_i = 1$ for all $i \in \C{Z}$.
\end{itemize}
\end{itemize}
\end{corollary}
\smallskip

\begin{proof}
Throughout the proof,
we let $\g{\mu}$, $\C{F}$, $\C{L}$ and $\C{U}$ denote
$\g{\mu}(\m{x}, \lambda)$, $\C{F}(\m{x})$, $\C{L}(\m{x})$, and $\C{U}(\m{x})$
respectively, where $\m{x}$ is a local minimizer for (\ref{Q}) and
the pair $(\m{x},\lambda)$ satisfies the first-order optimality
conditions (\ref{KT}).
To begin, suppose that $\m{x}$ is a strict local minimizer of (\ref{Q}).
That is, $f(\m{y}) > f(\m{x})$ when $\m{y}$ is a feasible point near $\m{x}$.
If ${\cal F}$ has at least two elements,
then by (P2) of Theorem \ref{opttheorem},
$d_{ii} + d_{jj} = 2a_{ij}$ for each $i$ and $j \in {\cal F}$.
Since the first-order optimality conditions (\ref{KT}) hold at $\m{x}$,
it follows from (\ref{h53}) that
\begin{equation}\label{constant_f}
f(\m{x}+\alpha (\m{e}_i - \m{e}_j)) = f(\m{x})
\end{equation}
for all $\alpha$.
Since this violates the assumption that $\m{x}$ is a strict local minimizer,
we conclude that $|{\cal F}| \le 1$.
If $\m{1}\tr\m{x} = u$ or $\m{1}\tr\m{x} = l$, then
since $u$ and $l$ are integers,
it is not possible for $\m{x}$ to have just one fractional component.
Consequently, ${\cal F}$ is empty.
If $l < \m{1}\tr\m{x} < u$, then by complementary slackness, $\lambda = 0$.
Suppose that $|{\cal F}| = 1$ and $i \in {\cal F}$.
By (P4) of Corollary \ref{optcorollary}, $d_{ii} = 0$.
Again, by (\ref{T}) it follows that
\[
f(\m{x}+\alpha \m{e}_i) = f(\m{x})
\]
for all $\alpha$.
This violates the assumption that $\m{x}$ is a strict local
minimizer of (\ref{Q}).
Hence, $\C{F}$ is empty.

By the first-order conditions (\ref{KT}), there exists $\lambda$ such that
\begin{equation}\label{h84}
\mu_i (\m{x},\lambda) \ge 0 \ge \mu_j (\m{x},\lambda)
\end{equation}
for all $i \in {\cal L}$ and $j \in {\cal U}$.
If this inequality  becomes an equality for
some $i \in {\cal L}$ and $j \in {\cal U}$, then
$\mu_i = 0 = \mu_j$,
and by (P3) of Corollary \ref{optcorollary},
we have $d_{ii} + d_{jj} = 2a_{ij}$.
Again, (\ref{constant_f}) violates the assumption that $\m{x}$
is a strict local minimizer.
Hence, one of the inequalities in (\ref{h84}) is strict.
The $\lambda$ on each side of (\ref{h84}) is cancelled to
obtain (C2).

Suppose that $l < u$, $\lambda = 0$,
and $\C{Z} := \{ i : \nabla f (\m{x})_i = 0 \} \ne \emptyset$.
When $\lambda = 0$, we have $\g{\mu}(\m{x},0) = \nabla f (\m{x})$.
Hence, $\C{Z} = \{ i : \mu_i (\m{x},0) = 0 \} \ne \emptyset$.
It follows from (P4) that in any of the cases (a)--(c),
we have $d_{ii} = 0$.
In particular, if $l < \m{1}\tr\m{x} < u$, then by
(\ref{T}), we have $f(\m{x}+\alpha \m{e}_i) = f(\m{x})$
for all $\alpha$.
Again, this violates the assumption that $\m{x}$ is a strict
local minimum.
Similarly, if for some $i \in \C{Z}$, either $x_i > 0$
and $\m{1}\tr\m{x} = u$ or $x_i < 1$ and $\m{1}\tr\m{x} = l$,
the identity $f(\m{x}+\alpha \m{e}_i) = f(\m{x})$
implies that we violate the strict local optimality of $\m{x}$.
This establishes (C3).

Conversely, suppose that $\m{x}$ is a local minimizer and (C1)--(C3) hold.
We will show that
\begin{equation}\label{st}
\nabla f(\m{x})\m{y} > 0 \mbox{ whenever } \m{y} \ne \m{0}
\mbox{ and } \m{x} + \m{y} \mbox{ feasible in } (\ref{Q}).
\end{equation}
As a result, by the mean value theorem,
$f(\m{x}+\m{y}) > f(\m{x})$ when $\m{y}$ is sufficiently small.
Hence, $\m{x}$ is a strict local minimizer.

When $\m{x} + \m{y}$ is feasible in (\ref{Q}), we have
\begin{equation}\label{sign}
y_i \ge 0 \mbox{ for all } i \in \C{L} \mbox{ and }
y_i \le 0 \mbox{ for all } i \in \C{U}.
\end{equation}
%
By the first-order optimality condition (\ref{KT}),
$\mu_i \ge 0$ for all $i \in \C{L}$ and
$\mu_i \le 0$ for all $i \in \C{U}$.
Hence, we have
\begin{equation}\label{xyz}
(\nabla f(\m{x}) + \lambda\m{1}\tr) \m{y} =
\g{\mu}\tr \m{y} = \sum_{i \in \C{L}} \mu_i y_i +
\sum_{i \in \C{U}} \mu_i y_i \ge 0 .
\end{equation}

We now consider three cases.

First, suppose that $\m{1}\tr\m{y} = 0$ and $\m{y} \ne \m{0}$.
By (C1) ${\cal F}$ is empty and
hence, by (\ref{sign}), $y_i > 0$ for some $i \in \C{L}$ and
$y_j < 0$ for some $j \in \C{U}$.
After adding $\lambda$ to each side in the inequality in (C2),
it follows that either
\begin{equation}\label{e1}
\min_{i \in \C{L}} \mu_i \ge 0 > \max_{j \in \C{U}} \mu_j
\end{equation}
or
\begin{equation}\label{e2}
\min_{i \in \C{L}} \mu_i > 0 \ge \max_{j \in \C{U}} \mu_j .
\end{equation}
Combining (\ref{xyz}), (\ref{e1}), and (\ref{e2})
gives $\nabla f(\m{x})\m{y} \ge \mu_i y_i - \mu_j y_j > 0$
since either $\mu_i > 0$ or $\mu_j < 0$, and $y_i > 0 > y_j$.

Second, suppose that $\m{1}\tr\m{y} \ne 0$ and $\lambda \ne 0$.
To be specific, suppose that $\lambda > 0$.
By complementary slackness, $\m{1}\tr\m{x} = u$.
Since $\m{x} + \m{y}$ is feasible in (\ref{Q}) and
$\m{1}\tr\m{y} \ne 0$, we must have $\m{1}\tr\m{y} < 0$.
Hence, by (\ref{xyz}), $\nabla f(\m{x})\m{y} > 0$.
The case $\lambda < 0$ is similar.

Finally, consider the case $\m{1}\tr \m{y} \ne 0$ and $\lambda = 0$.
In this case, we must have $l < u$.
If the set $\C{Z}$ in (C3) is empty, then
$\nabla f(\m{x})_i = \mu_i \ne 0$ for all $i$,
and by (\ref{xyz}), $\nabla f(\m{x})\m{y} > 0$.
If $\C{Z} \ne \emptyset$, then by (C3), either
$\m{1}\tr \m{x} = u$ and $x_i = 0$ for all $i \in \C{Z}$ or
$\m{1}\tr \m{x} = l$ and $x_i = 1$ for all $i \in \C{Z}$.
To be specific, suppose that
$\m{1}\tr \m{x} = u$ and $x_i = 0$ for all $i \in \C{Z}$.
Again, since $\m{x} + \m{y}$ is feasible in (\ref{Q}) and
$\m{1}\tr\m{y} \ne 0$, we have $\m{1}\tr\m{y} < 0$.
If $\C{U} = \emptyset$, then $\m{x} = \m{0}$ since $\C{F} = \emptyset$.
Since $\m{1}\tr\m{y} < 0$, we contradict the feasibility of
$\m{x} + \m{y}$.
Hence, $\C{U} \ne \emptyset$.
Since $\m{1}\tr\m{y} < 0$, there exists $j \in \C{U}$ such that $y_j < 0$.
Since $\C{Z} \subset \C{L}$, it follows from (\ref{e1}) that $\mu_j < 0$.
By (\ref{xyz}) $\nabla f(\m{x})\m{y} \ge \mu_j y_j > 0$.
The case $\m{1}\tr \m{x} = l$ and $x_i = 1$ for all $i \in \C{Z}$ is similar.
This completes the proof of (\ref{st}), and the corollary has been
established.
\end{proof}

\smallskip

\section{Numerical results}
\label{numerics}
We investigate the performance of the branch and bound algorithm
based on the lower bounds in Section \ref{LowerBound} using a series
of test problems. The codes were written in C and the experiments
were conducted on an Intel Xeon Quad-Core X5355 2.66 GHz computer
using the Linux operating system. Only one of the 4 processors was
used in the experiments. To evaluate the lower bound, we solve
(\ref{Ltau}) by the gradient projection method with an exact
linesearch and Barzilai-Borwein steplength \cite{bb88}. The stopping
criterion in our experiments was
\[
\| P(\m{x}_k - \m{g}_k) - \m{x}_k \| \le 10^{-4},
\]
where $P$ denotes the projection onto the feasible set and $\m{g}_k$
is the gradient of the objective function at $\m{x}_k$.
The solution of the semidefinite programming problem (\ref{sdp})
was obtained using Version 6.0.1 of the CSDP code
\cite{Borchers99} available at
\begin{center}
\medskip
$ \mbox{https://projects.coin-or.org/Csdp/} $
\medskip
\end{center}
We compare the performance of our algorithm with results reported by
Karisch, Rendl, and Clausen in \cite{Karisch00} and by Sensen in
\cite{Sensen01}. Since these earlier results were obtained on
different computers, we obtained estimates for the corresponding
running time on our computer using the LINPACK benchmarks
\cite{Dongarra08}. Since our computer is roughly 30 times faster
than the HP~9000/735 used in \cite{Karisch00} and it is roughly 7
times faster than the Sun UltrSPARC-II 400Mhz machine used in
\cite{Sensen01}, the earlier CPU times were divided by 30 and 7
respectively to obtain the estimated running time on our computer.
Note that the same interior-point algorithm that we use, which is
the main routine in the CSDP code, was used to solve the
semidefinite relaxation in \cite{Karisch00}.

The test problems were based on the graph bisection problem
where $l = u = n/2$.
Two different data sets were used for the $\m{A}$ matrices
in the numerical experiments.
Most of the test problems came from the library of
Brunetta, Conforti, and Rinaldi \cite{bcr97} which is available at
\begin{center}
\medskip $\mbox{ftp://ftp.math.unipd.it/pub/Misc/equicut}$.
\medskip
\end{center}
Some of the test matrices were from the UF Sparse Matrix
Library maintained by Timothy Davis:
\begin{center}
\medskip $ \mbox{http://www.cise.ufl.edu/research/sparse/matrices/} $
\medskip
\end{center}
Since this second set of matrices is not directly connected with graph
partitioning, we create an $\m{A}$ for graph partitioning as follows:
If the matrix $\m{S}$ from the library was symmetric,
then $\m{A}$ was the adjacency matrix defined as follows:
the diagonal of $\m{A}$ is zero,
$a_{ij} = 1$ if $s_{ij} \ne 0$, and $a_{ij} = 0$ otherwise.
If $\m{S}$ was not symmetric,
then $\m{A}$ was the adjacency matrix of $\m{S}\tr\m{S}$.

\subsection{Lower bound comparison}
Our numerical study begins with a comparison of the
lower bound of Section \ref{mineig} based on the minimum
eigenvalue of $\m{A}+\m{D}$ and the best affine underestimate,
and the lower bound of Section \ref{SDP} based on semidefinite programming.
We label these two lower bounds $LB_1$ and $LB_2$ respectively.
In Table \ref{tab1}, the first 5 graphs correspond to matrices
from the UF Sparse Matrix Library, while the next 5 graphs were from
the test set of Brunetta, Conforti, and Rinaldi.
The column labeled ``Opt'' is the minimum cut and while $n$ is the problem
dimension.
The numerical results indicate that the lower bound $LB_2$ based on
semidefinite programming is generally better (larger) than $LB_1$.
In Table \ref{tab1} the best lower bound is highlighted in bold.
Based on these results,
we use the semidefinite programming-based lower bound in the
numerical experiments which follow.
\begin{table}[h!]
\caption{Comparison of two lower bounds}

\begin{center}
\begin{tabular}
{|l|r|r|r|r|}
\hline Graph & $n$ & $LB_1$ & $LB_2$ & Opt  \\

\hline Tina\_Discal & 11 & 0.31 & {\bf 0.86} & 12 \\
\hline jg1009 & 9 & 1.55 & {\bf 1.72} & 16 \\
\hline jg1011 & 11 & {\bf 1.48} & 0.94 & 24 \\
\hline Stranke94 & 10 & 1.76 & {\bf 1.77} & 24 \\
\hline Hamrle1 & 32 & -1.93 & {\bf 1.12} & 17 \\
\hline 4x5t & 20 & -21.71 & {\bf 5.43} & 28 \\
\hline 8x5t & 40 & -16.16 & {\bf 2.91} & 33 \\
\hline t050 & 30 & 0.90 & {\bf 18.54} & 397 \\
\hline 2x17m & 34 & {\bf 1.33} & 1.27 & 316 \\
\hline s090 & 60 & -9.84 & {\bf 13.10} & 238 \\
\hline
\end{tabular}
\label{tab1}
\end{center}
\end{table}

\subsection{Algorithm performance}
Unless stated otherwise,
the remaining test problems
came from the library of Brunetta, Conforti, and Rinaldi \cite{bcr97}.
Table \ref{tab6} gives results for matrices associated with the
finite element method \cite{DeSouza94}.
The three methods are labeled CQB (our convex quadratic branch and bound
algorithm), KRC (algorithm of Karisch, Rendl, and Clausen \cite{Karisch00}),
and SEN (algorithm of Sensen \cite{Sensen01}).
``$n$'' is the problem dimension,
``\%'' is the percent of nonzeros in the matrix, and
``$\#$~nodes'' is the number of nodes in the branch and bound tree.
The CPU time is given in seconds.
The best time is highlighted in bold.
As can be seen in Table \ref{tab6}, CQB was fastest in 6 out of the
10 problems even though the number of nodes in the branch and bound
tree was much larger.
Thus both KRC and SEN provided much tighter relaxations,
however, the time to solve their relaxed problems was much larger than
the time to optimize our convex quadratics.

\begin{table}[b]
\caption{Mesh Instances}
\begin{center}
\begin{tabular}{|l|c|c|rr|rr|rr|}
\hline
\multicolumn{3}{|c}{} & \multicolumn{2}{c}{CQB} & \multicolumn{2}{c}{KRC} &
\multicolumn{2}{c|}{SEN} \\
\hline
graph &  $n$   & $\%$  &{\small$\#$}nodes & time &{\small $\#$}nodes & time &{\small$\#$}nodes & time \\
\hline
m4    & 32  & 10  & 22    & 0.05    & 1   & {\bf 0.03}  & 1 & 0.14\\
ma    & 54  & 5   & 8     & 0.16    & 1   & {\bf 0.10}  & 1 & 0.28\\
me    & 60  & 5   & 13    & 0.20    & 1   & {\bf 0.13}  & 1 & 0.28\\
m6    & 70  & 5   & 205   & {\bf 0.47}    & 1   & 1.23  & 1 & 1.43\\
mb    & 74  & 4   & 95    & {\bf 0.43}    & 1   & 0.98  & 1 & 1.14\\
mc    & 74  & 5   & 412   & {\bf 0.52}    & 1   & 1.53  & 1 & 1.43\\
md    & 80  & 4   & 101   & {\bf 0.55}    & 1   & 0.96  & 1 & 1.28\\
mf    & 90  & 4   & 99    & {\bf 0.79}    & 1   & 0.80  & 1 & 1.85\\
m1    & 100 & 3   & 200   & {\bf 1.04}    & 15  & 36.50 & 1 & 3.00\\
m8    & 148 & 2   & 3516  & 6.62    & 1   & 10.70 & 1 & {\bf 4.14}\\
\hline
\end{tabular}
\label{tab6}
\end{center} \end{table}

Table \ref{tab7} gives results for
compiler design problems \cite{Ferreira98,JMN93}.
For this test set, KRC was fastest in 3 out of 5 test problems.
Note though that the times for CQB were competitive with KRC.

\begin{table}[h!b!p!]
\caption{Compiler Design}
\begin{center}
\begin{tabular}{|l|c|c|rr|rr|rr|}
    \hline
\multicolumn{3}{|c}{} & \multicolumn{2}{c}{CQB} & \multicolumn{2}{c}{KRC} &
\multicolumn{2}{c|}{SEN} \\
    \hline
graph &  $n$   & $\%$  &{\small$\#$}nodes & time &{\small $\#$}nodes & time &{\small$\#$}nodes & time \\
    \hline
cd30    & 30  & 13  & 11    & 0.05    & 1   & 0.03  & 1 & {\bf 0.00}\\
cd45    & 45  & 10  & 35    & 0.27    & 1   & {\bf 0.23}  & 1 & 0.57\\
cd47a   & 47  & 9   & 45    & 0.34    & 1   & {\bf 0.33}  & 7 & 1.00\\
cd47b   & 47  & 9   & 67    & {\bf 0.29} & 35  & 3.73  & 3 & 1.43\\
cd61    & 61  & 10  & 95    & 0.86    & 1 & {\bf 0.67}  & 6 & 6.00\\
    \hline
\end{tabular}
\label{tab7}
\end{center}
\end{table}

Table \ref{tab8} gives results for binary de Bruijn graphs which
arise in applications related to parallel computer architecture
\cite{Collins92,Feldmann97}. These graphs are constructed by the
following procedure. We first build a directed graph using the
Mathematica command:
\begin{center}
\medskip
\verb"A = TableForm[ToAdjacencyMatrix[DeBruijnGraph[2, n]]]"
\medskip
\end{center}
To obtain the graph partitioning test problem, we add the
Mathematica generated matrix to its transpose and set the diagonal
to 0. For this test set, SEN had by far the best performance.

\begin{table}[h!b!p!]
\caption{de Bruijn Networks}
\begin{center}
\begin{tabular}{|l|c|c|rr|rr|rr|}
    \hline
\multicolumn{3}{|c}{} & \multicolumn{2}{c}{CQB} & \multicolumn{2}{c}{KRC} &
\multicolumn{2}{c|}{SEN} \\
    \hline
graph &  $n$   & $\%$  &{\small$\#$}nodes & time &{\small $\#$}nodes & time &{\small$\#$}nodes & time \\
    \hline
debr5    & 32  & 12  & 57       & 0.11    & 3   & 0.20  & 1 & {\bf 0.00}\\
debr6    & 64  & 6   & 7327     & 2.25    & 55  & 15.63 & 1 & {\bf 1.00}\\
debr7    & 128 & 3   & 16140945 & 1:22:45 & 711 & 46:36 & 1 & {\bf 10.28}\\
    \hline
\end{tabular}
\label{tab8}
\end{center}
\end{table}

Table \ref{tab2} gives results for toroidal grid graphs.
These graphs are connected with an $h \times k$ grid,
the number of vertices in the graph is $n = hk$ and
there are $2hk$ edges whose weights are chosen from a
uniform distribution on the interval $[1, 10]$.
Since Sensen did not solve either this test set,
or the remaining test sets, we now compare between CQB and KRC.
We see in Table \ref{tab2} that CQB was faster than KRC in
9 of the 10 toroidal grid cases.

\begin{table}[h!b!p!]
\caption{Toroidal Grid: a weighted $h \times k$ grid with $hk$ vertices
and $2hk$ edges that received integer weights uniformly drawn from [1,10]}
\begin{center}
\begin{tabular}{|l|c|c|rr|rr|}
    \hline
\multicolumn{3}{|c}{} & \multicolumn{2}{c}{CQB} & \multicolumn{2}{c|}{KRC} \\
    \hline
graph &  $n$   & $\%$  &{\small$\#$}nodes & time &{\small $\#$}nodes & time \\
    \hline
4x5t     & 20  & 21  & 13     & {\bf 0.01}    & 1   & 0.03  \\
6x5t     & 30  & 14  & 46     & {\bf 0.05}    & 1   & 0.10  \\
8x5t     & 40  & 10  & 141    & {\bf 0.16}    & 1   & 0.20  \\
21x2t    & 42  & 10  & 18     & {\bf 0.02}    & 1   & 0.17  \\
23x2t    & 46  & 9   & 78     & {\bf 0.15}    & 33  & 4.16  \\
4x12t    & 48  & 9   & 69     & {\bf 0.17}    & 3   & 0.56  \\
5x10t    & 50  & 8   & 129    & 0.24    & 1   & {\bf 0.20}  \\
6x10t    & 60  & 7   & 992    & {\bf 0.54}    & 43  & 11.66 \\
7x10t    & 70  & 6   & 844    & {\bf 0.68}    & 47  & 19.06 \\
10x8t    & 80  & 5   & 420    & {\bf 0.91}    & 45  & 31.46 \\
    \hline
\end{tabular}
\label{tab2}
\end{center}
\end{table}

Table \ref{tab3} gives results for mixed grid graphs.
These are complete graphs associated with an planar
$h \times k$ planar grid; the edges in the planar grid received integer
weights uniformly drawn from [1,100], while all the other edges needed to
complete the graph received integer weights uniformly drawn from [1,10].
For these graphs, KRC was much faster than CQB.
Notice that the graphs in this test set are completely dense.
One trend that is seen in these numerical experiments is that as
the graph density increases, the performance of CQB relative to
the other methods degrades.

\begin{table}[h!b!p!]
\caption{Mixed Grid Graphs}
\begin{center}
\begin{tabular}{|l|c|c|rr|rr|rr|}
    \hline
\multicolumn{3}{|c}{} & \multicolumn{2}{c}{CQB} & \multicolumn{2}{c|}{KRC} \\
    \hline
graph &  $n$   & $\%$  &{\small$\#$}nodes & time &{\small $\#$}nodes & time  \\
    \hline
2x10m    & 20  & 100  & 150     & {\bf 0.03}    & 1   & {\bf 0.03}  \\
6x5m     & 30  & 100  & 2476    & 0.20    & 1   & {\bf 0.03}  \\
2x17m    & 34  & 100  & 42410   & 2.12    & 21  & {\bf 0.96}  \\
10x4m    & 40  & 100  & 51713   & 3.74    & 2   & {\bf 0.06}  \\
5x10m    & 50  & 100  & 3588797 & 296.19  & 1   & {\bf 0.06}  \\
    \hline
\end{tabular}
\label{tab3}
\end{center}
\end{table}

Results for planar grid graph are given in Table \ref{tab4}.
These graphs are associated with an $h \times k$ grid.
There are $hk$ vertices and $2hk - h - k$ edges whose
weights are integers uniformly drawn from [1,10].
For this relatively sparse test set,
CQB was faster in 7 out of 10 problems.

\begin{table}[h!b!p!]
\caption{Planar Grid}
\begin{center}
\begin{tabular}{|l|c|c|rr|rr|rr|}
    \hline
\multicolumn{3}{|c}{} & \multicolumn{2}{c}{CQB} & \multicolumn{2}{c|}{KRC} \\
    \hline
graph &  $n$   & $\%$  &{\small$\#$}nodes & time &{\small $\#$}nodes & time  \\
    \hline
10x2g    & 20  & 15  & 10     & {\bf 0.01}    & 1   & 0.03  \\
5x6g     & 30  & 11  & 44     & {\bf 0.05}    & 1   & 0.10  \\
2x16g    & 32  & 9   & 23     & {\bf 0.06}    & 1   & 0.13  \\
18x2g    & 36  & 8   & 19     & 0.08    & 1   & {\bf 0.06}  \\
2x19g    & 38  & 8   & 53     & {\bf 0.29}    & 49  & 1.83  \\
5x8g     & 40  & 9   & 24     & 0.08    & 1   & {\bf 0.06}  \\
3x14g    & 42  & 8   & 31     & {\bf 0.14}    & 5   & 0.60  \\
5x10g    & 50  & 7   & 178    & 0.34    & 1   & {\bf 0.30}  \\
6x10g    & 60  & 6   & 224    & {\bf 0.35}    & 57  & 10.63 \\
7x10g    & 70  & 5   & 271    & {\bf 0.63}    & 61  & 18.56 \\
    \hline
\end{tabular}
\label{tab4}
\end{center}
\end{table}

Table \ref{tab5} gives results for randomly generated graphs.
For these graphs, the density is first fixed and then the edges
are assigned integer weights uniformly drawn from [1,10].
For this test set, CQB is fastest in 11 of 20 cases.
Again, observe that the relative performance of CQB degrades as
the density increases, mainly due to the large number of nodes in
the branch and bound tree.

\begin{table}[h!b!p!]
\caption{Randomly Generated Graphs}
\begin{center}
\begin{tabular}{|l|c|c|rr|rr|rr|}
    \hline
\multicolumn{3}{|c}{} & \multicolumn{2}{c}{CQB} & \multicolumn{2}{c|}{KRC} \\
    \hline
graph &  $n$   & $\%$  &{\small$\#$}nodes & time &{\small $\#$}nodes & time  \\
    \hline
v090    & 20  & 10   & 12     & {\bf 0.01}    & 1   & 0.03  \\
v000    & 20  & 100  & 952    & {\bf 0.02}    & 1   & 0.03  \\
t090    & 30  & 10   & 10     & 0.05    & 1   & {\bf 0.03}  \\
t050    & 30  & 50   & 5081   & {\bf 0.32}    & 17  & 0.73  \\
t000    & 30  & 100  & 122670 & 3.79    & 3   & {\bf 0.20}  \\
q090    & 40  & 10   & 89     & 0.14    & 1   & {\bf 0.13}  \\
q080    & 40  & 20   & 914    & {\bf 0.24}    & 31  & 2.30  \\
q030    & 40  & 70   & 554652 & 32.23   & 23  & {\bf 2.06}  \\
q020    & 40  & 80   & 1364517& 72.58   & 7   & {\bf 0.83}  \\
q010    & 40  & 90   & 4344123& 217.16  & 13  & {\bf 1.36}  \\
q000    & 40  & 100  & 8186984& 380.72  & 1   & {\bf 0.13}  \\
c090    & 50  & 10   & 397    & {\bf 0.29}    & 1   & 0.33  \\
c080    & 50  & 20   & 14290  & {\bf 2.20}    & 45  & 6.13  \\
c070    & 50  & 30   & 136290 & 15.70   & 49  & {\bf 8.06}  \\
c030    & 50  & 70   & 22858729&2756.26 & 51  & {\bf 5.46}  \\
c290    & 52  & 10   & 340    & {\bf 0.34}    & 1   & 0.40  \\
c490    & 54  & 10   & 1443   & {\bf 0.54}    & 15  & 3.30  \\
c690    & 56  & 10   & 3405   & {\bf 0.82}    & 3   & 1.00  \\
c890    & 58  & 10   & 13385  & {\bf 2.66}    & 71  & 17.53 \\
s090    & 60  & 10   & 8283   & {\bf 2.01}    & 37  & 9.90  \\
    \hline
\end{tabular}
\label{tab5}
\end{center}
\end{table}

\section{Conclusions}
\label{conclusions}
An exact algorithm is presented for solving the graph partitioning
problem with upper and lower bounds on the size of each set
in the partition.
The algorithm is based on a continuous quadratic programming
formulation of the discrete partitioning problem.
We show how to transform a feasible $\m{x}$ for
the graph partitioning QP (\ref{Q}) to a binary feasible point $\m{y}$
with an objective function value which satisfies $f(\m{y}) \le f(\m{x})$.
The binary feasible point corresponds to a partition of the graph
vertices and $f(\m{y})$ is the weight of the cut edges.
At any stationary point of (\ref{Q}) which is not a local minimizer,
Proposition \ref{descent_direction} provides a descent direction that
can be used to strictly improve the objective function value.

In the branch and bound algorithm,
the objective function is decomposed into the sum of
a convex and a concave part.
A lower bound for the objective function is achieved by replacing
the concave part by an affine underestimate.
Two different decompositions were considered, one based on the
minimum eigenvalue of the matrix in the objective function,
and the other based on  the solution to a semidefinite programming problem.
The semidefinite programming approach generally led to much tighter
lower bounds.
In a series of numerical experiments, the new algorithm CQB
(convex quadratic branch and bound) was
competitive with state-of-the-art partitioning methods;
the relative performance of CQB was better for sparse graphs
than for dense graphs.


\begin{thebibliography}{10}

\bibitem{bb88}
{\sc J.~Barzilai and J.~M. Borwein}, {\em Two point step size gradient
  methods}, {IMA} J. Numer. Anal., 8 (1988), pp.~141--148.

\bibitem{Borchers99}
{\sc B.~Borchers}, {\em {CSDP, A C} library for semidefinite programming},
  Optimization Methods and Software, 11(1) (1999), pp.~613--623.

\bibitem{bcr97}
{\sc L.~Brunetta, M.~Conforti, and G.~Rinaldi}, {\em A branch-and-cut algorithm
  for the equicut problem}, Math. Program., 78 (1997), pp.~243--263.

\bibitem{Catalyurek-hyper}
{\sc U.~V. Cataly\"{u}rek and C.~Aykanat}, {\em Hypergraph-partitioning based
  decomposition for parallel sparse-matrix vector multiplication}, IEEE
  Transaction on Parallel and Distributed Systems, 10 (1999), pp.~673--693.

\bibitem{Collins92}
{\sc O.~Collins, S.~Dolinar, R.~McEliece, and F.~Pollara}, {\em A {VLSI}
  decomposition of the de {B}ruijn graph}, J. ACM, 39 (1992), pp.~931--948.

\bibitem{Zoltan06ipdps}
{\sc K.~Devine, E.~Boman, R.~Heaphy, R.~Bisseling, and U.~Catalyurek}, {\em
  Parallel hypergraph partitioning for scientific computing}, in Proc. of 20th
  International Parallel and Distributed Processing Symposium (IPDPS'06), IEEE,
  2006.

\bibitem{Dongarra08}
{\sc J.~J. Dongarra}, {\em Performance of various computers using standard
  linear equations software}, Tech. Report cs-89-85, University of Tennessee,
  Knoxville, TN, USA, 2008.

\bibitem{FalknerRendlWolkowicz94}
{\sc J.~Falkner, F.~Rendl, and H.~Wolkowicz}, {\em A computational study of
  graph partitioning}, Math. Programming, 66 (1994), pp.~211--240.

\bibitem{Feldmann97}
{\sc R.~Feldmann, B.~Monien, P.~Mysliwietz, and S.~Tsch{\"{o}}ke}, {\em A
  better upper bound on the bisection width of de {B}ruijn networks}, in STACS
  97, vol.~1200 of Lecture Notes in Computer Science, Springer, Berlin, 1997,
  pp.~511--522.

\bibitem{Ferreira98}
{\sc C.~E. Ferreira, A.~Martin, C.~C. de~Souza, R.~Weismantel, and L.~A.
  Wolsey}, {\em The node capacitated graph partitioning problem: {A}
  computational study}, Math. Programming, 81 (1998), pp.~229--256.

\bibitem{Garey76}
{\sc M.~R. Garey, D.~S. Johnson, and L.~Stockmeyer}, {\em Some simplified
  {$NP$}-complete graph problems}, Theoretical Computer Science, 1 (1976),
  pp.~237--267.

\bibitem{GilbertMillerTeng98}
{\sc J.~R. Gilbert, G.~L. Miller, and S.~H. Teng}, {\em Geometric mesh
  partitioning: {Implementation} and experiments}, {SIAM} J. Sci. Comput., 19
  (1998), pp.~2091--2110.

\bibitem{GrigoriBomanDonfackDavis08}
{\sc L.~Grigori, E.~Boman, S.~Donfack, and T.~A. Davis}, {\em Hypergraph-based
  unsymmetric nested dissection ordering for sparse {LU} factorization}, {SIAM}
  J. Sci. Comput.,  (2008).
\newblock under submission.

\bibitem{HagerKrylyuk99}
{\sc W.~W. Hager and Y.~Krylyuk}, {\em Graph partitioning and continuous
  quadratic programming}, {SIAM} J. Disc. Math., 12 (1999), pp.~500--523.

\bibitem{HagerKrylyuk02}
\leavevmode\vrule height 2pt depth -1.6pt width 23pt, {\em Multiset graph
  partitioning}, Math. Meth. Oper. Res., 55 (2002), pp.~1--10.

\bibitem{HagerPhan09}
{\sc W.~W. Hager and D.~T. Phan}, {\em An ellipsoidal branch and bound
  algorithm for global optimization}, {SIAM} J. Optim., 20 (2009),
  pp.~740--758.

\bibitem{hendrickson94chaco}
{\sc B.~Hendrickson and R.~Leland}, {\em The {Chaco user's guide - Version 2.0,
  Sandia National Laboratories, Technical Report SAND94-2692}}, 1994.

\bibitem{HendricksonLeland95}
\leavevmode\vrule height 2pt depth -1.6pt width 23pt, {\em An improved spectral
  graph partitioning algorithm for mapping parallel computations}, {SIAM} J.
  Sci. Comput., 16 (1995), pp.~452--469.

\bibitem{Hendrickson}
\leavevmode\vrule height 2pt depth -1.6pt width 23pt, {\em A multilevel
  algorithm for partitioning graphs}, in Proc. Supercomputing '95, ACM, Nov.
  1995.

\bibitem{HPT95}
{\sc R.~Horst, P.~M. Pardalos, and N.~V. Thoai}, {\em Introduction to Global
  Optimization}, Kluwer Academic Publishers, Dordrecht, Holland, 1995.

\bibitem{JMN93}
{\sc E.~L. Johnson, A.~Mehrotra, and G.~L. Nemhauser}, {\em Min-cut
  clustering}, Math. Program., 62 (1993), pp.~133--151.

\bibitem{Johnson93}
{\sc T.~Johnson}, {\em A concurrent dynamic task graph}, in Proc. 1993 Intl.
  Conf. on Parallel Processing, 1993.
\newblock (TR-93-011, CISE Dept., Univ. of Florida).

\bibitem{Karisch00}
{\sc S.~E. Karisch, F.~Rendl, and J.~Clausen}, {\em Solving graph bisection
  problems with semidefinite programming}, INFORMS Journal on Computing, 12
  (2000), pp.~177--191.

\bibitem{KarypisKumar98e}
{\sc G.~Karypis and V.~Kumar}, {\em A fast and high quality multilevel scheme
  for partitioning irregular graphs}, {SIAM} J. Sci. Comput., 20 (1998),
  pp.~359--392.

\bibitem{KarypisKumar99b}
\leavevmode\vrule height 2pt depth -1.6pt width 23pt, {\em Parallel multilevel
  k-way partitioning scheme for irregular graphs}, {SIAM} Review, 41 (1999),
  pp.~278--300.

\bibitem{KarypisKumar00}
\leavevmode\vrule height 2pt depth -1.6pt width 23pt, {\em Multilevel k-way
  hypergraph partitioning}, VLSI Design, 11 (2000), pp.~285--300.

\bibitem{Lengauer}
{\sc T.~Lengauer}, {\em Combinatorial Algorithms for Integrated Circuit
  Layout}, John Wiley, Chichester, 1990.

\bibitem{Mitchell01}
{\sc J.~Mitchell}, {\em Branch-and-cut for the $k$-way equipartition problem},
  tech. report, Department of Mathematical Sciences, Rensselaer Polytechnic
  Institute, 2001.

\bibitem{murty1987}
{\sc K.~G. Murty and S.~N. Kabadi}, {\em Some {NP}-complete problems in
  quadratic and linear programming}, Math. Program., 39 (1987), pp.~117--129.

\bibitem{pardalos91}
{\sc P.~M. Pardalos and S.~A. Vavasis}, {\em Quadratic programming with one
  negative eigenvalue is {NP}-hard}, J. Global Optim., 1 (1991), pp.~15--22.

\bibitem{PellegriniRomanAmestoy00}
{\sc F.~Pellegrini, J.~Roman, and P.~R. Amestoy}, {\em Hybridizing nested
  dissection and halo approximate minimum degree for efficient sparse matrix
  ordering}, Concurrency: Pract. Exp., 12 (2000), pp.~68--84.

\bibitem{Preis96theparty}
{\sc R.~Preis and R.~Diekmann}, {\em The {PARTY} partitioning library user
  guide - version 1.1}, 1996.

\bibitem{Sensen01}
{\sc N.~Sensen}, {\em Lower bounds and exact algorithms for the graph
  partitioning problem using multicommodity flows}, in Lecture Notes in
  Computer Science, vol. 2161, Springer-Verlag, 2001, pp.~391--403.

\bibitem{SoperWalshawCross04}
{\sc A.~J. Soper, C.~Walshaw, and M.~Cross}, {\em A combined multilevel search
  and multilevel optimization approach to graph-partition}, J. Global Optim.,
  29 (2004), pp.~225--241.

\bibitem{DeSouza94}
{\sc C.~Souza, R.~Keunings, L.~Wolsey, and O.~Zone}, {\em A new approach to
  minimizing the frontwidth in finite element calculations}, Computer Methods
  in Applied Mechanics and Engineering, 111 (1994), pp.~323--334.

\bibitem{Teng}
{\sc S.-H. Teng}, {\em Provably good partitioning and load balancing algorithms
  for parallel adaptive {N}-body simulation}, {SIAM} J. Sci. Comput., 19
  (1998), pp.~635--656.

\bibitem{WalshawCrossEverett97}
{\sc C.~Walshaw, M.~Cross, and M.~Everett}, {\em Parallel dynamic graph
  partitioning for adaptive unstructured meshs}, J. Parallel Distrib. Comput.,
  47 (1997), pp.~102--108.

\end{thebibliography}
\end{document}